\documentclass[english]{amsart}
\usepackage{amssymb,amsfonts,colordvi,color}
\usepackage[english]{babel}

\newtheorem{theorem}{Theorem}

\newtheorem{lemma}[theorem]{Lemma}
\newtheorem{claim}[theorem]{Claim}

\newtheorem{proposition}[theorem]{Proposition}
\newtheorem{definition}[theorem]{Definition}

\newcommand{\field}[1]{\mathbb{#1}}
\newcommand{\A}{\field{A}}
\newcommand{\C}{\field{C}}
\newcommand{\N}{\field{N}}
\newcommand{\Q}{\field{Q}}
\newcommand{\R}{\field{R}}

\newcommand{\ovl}[1]{\overline{#1}}
\newcommand{\bb}{{\bf b}}
\newcommand{\bB}{{\bf B}}

\newcommand{\bH}{{\bf H}}
\newcommand{\bk}{{\bf k}}
\newcommand{\bm}{{\bf m}}

\newcommand{\bp}{{\bf p}}
\newcommand{\bP}{{\bf P}}
\newcommand{\bS}{{\bf S}}

\newcommand{\bu}{{\bf u}}

\newcommand{\bw}{{\bf w}}
\newcommand{\bx}{{\bf x}}
\newcommand{\by}{{\bf y}}

\newcommand{\Cn}{{\C^n}}
\newcommand{\cO}{\mathcal{O}}
\newcommand{\cU}{\mathcal{U}}
\newcommand{\cV}{\mathcal{V}}
\newcommand{\dist}{{\rm dist}}

\newcommand{\du}{{\partial_u}}

\newcommand{\gm}{{\gamma}}

\newcommand{\kp}{{\kappa}}

\newcommand{\omg}{\omega}
\newcommand{\ovlX}{\ovl{X}}
\newcommand{\ovlY}{\ovl{Y}}

\newcommand{\sgm}{\sigma}
\newcommand{\tht}{{\theta}}
\newcommand{\ve}{\varepsilon}
\newcommand{\vp}{\varphi}

\title[local Lipschitz triviality of complex polynomials]{A note on the local Lipschitz triviality 
of values of complex polynomial functions}
\author{Alexandre Fernandes}
\address{Departamento de Matem\'atica, Universidade Federal do Cear\'a
(UFC), Campus do Pici, Bloco 914, Cep. 60455-760. Fortaleza-Ce,
Brasil}
\email{alex@mat.ufc.br}
\author{Vincent Grandjean}
\address{Departamento de Matem\'atica, Universidade Federal do Cear\'a
(UFC), Campus do Pici, Bloco 914, Cep. 60455-760. Fortaleza-Ce,
Brasil}
\email{vgrandjean@mat.ufc.br}
\author{Humberto Soares}
\address{Departamento de Matem\'atica, Universidade Federal do Piaui
(UFPi),  Campus Universit\'ario Ministro Petr\^onio Portella -  Cep. 64049-550. Teresina-Pi, Brasil}
\email{humberto@ufpi.edu.br}
\subjclass{14B05 32S15 }
\keywords{complex polynomials, regular value, bi-Lipschitz trivialization}
\thanks{A. Fernandes was partially supported by FUNCAP/CAPES/CNPq grant 304221/2017-1
\\
V. Grandjean was partially supported by FUNCAP/CAPES/CNPq grant 305614/2015-0
C.H. Soares was partially supported by CNPq grant 113058/2016-0
}
\begin{document}
\begin{abstract}
We address here the question of the bi-Lipschitz local triviality of a complex polynomial
function over a complex value. 

Our main results states that a non constant complex polynomial admits a locally bi-Lipschitz trivial 
value if and only if it is a polynomial in one complex variable.
\end{abstract}
\maketitle



\section{Introduction}
One of the results of Thom's famous \em Ensembles et Morphismes Stratifi\'es \em
\cite{Tho} is that, given any complex polynomial function $f:\Cn\mapsto \C$, there 
exists a smallest finite subset $B(f)$ of values, called the \em bifurcation set of $f$, \em 
such that the space $f^{-1}(\C\setminus B(f))$ is a smooth fiber bundle over $\C \setminus B(f)$ 
with model fiber $f^{-1}(a)$ for any value $a$ taken outside $B(f)$. 
The bifurcation locus $B(f)$ always contains the critical values of $f$, but may contain also regular values.
In particular for any value $a$ not in $B(f)$, we always find a neighbourhood $\cU$ of $a$ in $\C$ such that
$f^{-1}(\cU)$ is diffeomorphic, as a fiber bundle, to the trivial bundle $f^{-1}(a) \times \cU$. 
A value $a$ at which this local triviality condition
of the function $f$ is satisfied will be called \em typical value \em of the function.

On the other hand, given an algebraic family $\{S_t\}_{t\in P}$ of complex algebraic sets of $\Cn$ or $\bP\Cn$,
many results have been produced in the last fifty years to guarantee the local topological constancy of $S_t$ 
in a neighbourhood of a given parameter. Most often it is controlled by the regularity of a tailored stratification
of the parameter space. Any subset $S$ of a metric space $(E,d)$ is a  metric space in its own when equipped with the ambient metric: 
the distance between any pair of points of $S$ is taken in the ambient space $(E,d)$. Two subsets $S,S'$ of a metric space $(E,d)$ 
are \em bi-Lipschitz equivalent \em or have \em the same bi-Lipschitz type, \em if $S,S'$ are bi-Lipschitz homeomorphic metric spaces. 
As a consequence of Mostowski's result about the existence of Lipschitz stratification
of affine complex algebraic sets \cite{Mos}, there are finitely many \em Lipschitz types \em in the family $\{S_t\}_t$.

\medskip
In this note we would like to address the following problem of equi-singularity:

\smallskip\noindent
{\bf Local bi-Lipschitz triviality of $f$ nearby a value:} 
 \em Suppose $a$ is a typical value of a complex polynomial $f$. 
Can we find a neighbourhood $\cV$ of the value $a$ such that $f^{-1}(\cV)$ is bi-Lipschitz homeomorphic,
as a fiber bundle over  $\cV$, to $f^{-1}(a) \times \cV$ ? The open subset $f^{-1}(\cV)$ is equipped with the ambient 
Euclidean metric while the trivial bundle $f^{-1}(a) \times \cV$ is equipped with the product metric, that is
$\dist((\bx,t),(\bx',t')) = |\bx - \bx'| + |t-t'|.$ \em

\smallskip
Since $\Cn$ is not compact the smooth structure of the bundle $f^{-1}(\cV)$ does not \`a-priori imply that 
it is also bi-Lipschitz in the sense presented above. 
\\
A polynomial $f:\C^n\mapsto \C$ is a \em polynomial in 
a single variable, \em if there exist $n-1$ linearly independent vectors $\xi_2,\ldots,\xi_n,$ of $\C^n$
such that $\xi_j \cdot f \equiv 0$ for all $j$.
The main result we present here is the following :

\smallskip\noindent 
{\bf Theorem \ref{thm:main}.} \em A non constant complex polynomial $f:\C^n \mapsto \C$ admits a 
bi-Lipschitz-trivial value if and only if it is a polynomial in a single variable. \em

\medskip
An equivalent reformulation of Theorem \ref{thm:main} is the following: \em 
A complex polynomial $f:\C^n \mapsto \C$, for $n\geq 2$, is not a polynomial in a single variable,
if and only if for any values $s,t,$ outside of a finite subset of $\C$
\begin{center}
$\dist (f^{-1}(s),f^{-1}(t)) = 0.$ 
\end{center}
\em

\bigskip
The paper is organized as follows. Section 2 starts with a counter-example in two variables, namely 
$f(x,y) = xy$. Section 3 deals with the geometry at infinity of the levels of the function and 
Lemma \ref{lem:cone-tangent-1} gives a necessary condition on this geometry for the 
function to admit a bi-Lipschitz trivial value. Section 4 is the core of our arguments and 
deals with the main result in the plane case. The last Section gives the proof in the general case using 
an induction argument on the ambient dimension, which turns easy after the work in dimension $2$.
%
%
%
%
%
%
%
%
%
%
%
%
%
%
%
%
%
%
%
%
%
%
%
%
%
%
\section{local bi-Lipschitz triviality and counter-examples}
Let $f:\Cn \mapsto \C$ be a polynomial of degree $d\geq 0$.
For any subset $Z$ of $\C$, let us denote $X_Z := f^{-1}(Z)$, so that $X_t := \{f =t\}$.

\smallskip
As already mentioned in the introduction, a famous result of Thom \cite{Tho} asserts that the subset 
$X_{\C\setminus B(f)}$ is a smooth fiber bundle over $\C\setminus B(f)$ with model fiber $X_c$ for some (any) $c$ of 
$\C\setminus B(f)$. 
In particular, the family of levels $\{X_t\}_{t \in\C}$ has constant smooth type 
at any point of $\C\setminus B(f)$.
At any value lying in the \em bifurcation locus \em $B(f)$ of the function $f$, 
the smooth type of the corresponding level of the function changes. 

\medskip
Any non-empty subset $X_Z$ is a metric space when equipped with the ambient distance, that 
is taken in the Euclidean $\Cn$. 
\begin{definition}
A value $c$ taken by $f$ is called \em bi-Lipschitz trivial \em if there is a neighbourhood $\cV$ of $c$ in $\C$
such that $X_\cV$ is bi-Lipschitz homeomorphic, as a fiber bundle over $\cV$, to the trivial bundle $X_c\times\cV$.
\end{definition}
Again, the question we want to address is the following:

\smallskip\noindent
{\bf Question:} \em Let $c$ be any typical value. Does there exist a neighbourhood $\cV$
of $c$ such that $X_\cV$ is, in the sense of fiber bundles, bi-Lipschitz homeomorphic to the trivial 
bundle $X_c\times \cV$ ? \em

\medskip
Before getting into the problem itself, observe that being locally smoothly trivial as it is in a smooth fiber bundle
is a local condition in the base of the bundle as well as in the fiber. 
Asking local bi-Lipschitz-ness along the base, as we do, put some global constraint fiber-wise speaking on the 
trivializing mapping, since the model fiber $X_c$ is not compact.

\medskip
The next Lemma although obvious is key to this note
\begin{lemma}
Let $f:\C^n \mapsto \C$ be a polynomial. Let $c$ be a regular value taken by $f$
for which there exists $\ve>0$ such that 

(i) the open ball $\bB_\ve(c)$ consists only in regular values of $f$;

(ii) The function $f$ is bi-Lipschitz trivial over $\bB_\ve(c)$.

\smallskip
Then there exists $L_\ve > 1$ such that 
\smallskip
\begin{equation}\label{eq:distance-fibres}
\frac{1}{L_\ve}|s-t| \leq \dist(X_s,X_t) \leq L_\ve|s-t|  \mbox{ for all }  s,t\in\bB_\ve(c).
\end{equation}
\end{lemma}
\begin{proof}
This is a consequence of the definition of bi-Lipschitz triviality.
\end{proof}

\medskip
Let us consider the class of polynomials in a \em single variable: \em
Let $\lambda : \Cn \mapsto \C$ be a non-zero $\C$-affine function.
Let $g:\Cn\mapsto\C$ be a complex polynomial of the form $g = P\circ \lambda$ for  $P:\C\mapsto\C$ a non constant polynomial.
Each level of $g$ is a finite union of parallel hyperplanes and we can check that the the bifurcation values of $g$
are the critical values of $P$. It is straightforward to check that each regular value of the function $g$ is 
bi-Lipschitz trivial.

\smallskip
We follow with a counter-example.
\begin{proposition}\label{prop:xy}
Let $f:\C^2\mapsto\C$ defined as $f(x,y) = xy$. The function $f$ has no bi-Lipschitz trivial value.
\end{proposition}
\begin{proof}
The bifurcation values of the function $f$ reduces to the critical value $0$.
Let $s,t$ be two values of $\C\setminus 0$. Since the levels of the function
are graphs over $\C$, an obvious calculation gives: 
$$
\dist(X_s,X_t) = 0.
$$
\end{proof}
%
%
%
%
%
%
%
%
%
%
%
%
%
%
%
%
%
%
%
%
%
%
%
%
%
%
%
%
%
%
%
%
%
%
%
\section{Tangent cones at points at infinity}
We present here Lemma \ref{lem:cone-tangent-1}, most likely already known.  
It yields a necessary condition on the local geometry at infinity of the levels of a complex polynomial function, 
if the latter were to admit a bi-Lipschitz trivial value. 

\bigskip
Let $\bH :=\bP\Cn\setminus\Cn = \bP\C^{n-1}$ be the hyperplane at infinity.

Let $[\bx:z]$ be coordinates over $\bP\Cn$ such that $\bH:=\{z=0\}$.

Let $\cU_i$ be the affine chart of $\bP\Cn$  given by $x_i \neq 0$ and $\cU_0 = \Cn$ be the
affine chart $z\neq 0$

\medskip
Let $X$ be an affine and irreducible complex hypersurface of $\C^n$. 
Let $\ovlX$ be its projective closure in $\bP\Cn$ and let $X^\infty := \ovlX \cap \bH$
be its trace a infinity.
Let $F$ be a reduced homogeneous polynomial of degree $d$ such that $\ovlX = \{ F = 0\}$.
We can write 
$$
F[\bx:z] = f_d(\bx) + z f_{d-1}(\bx) + \ldots + z^df_0.
$$

\smallskip
Let $Y$ be another irreducible hypersurface of $\Cn$ whose projective closure $\ovl{Y}$ is
given by the equation $\{G = 0\}$ for an irreducible polynomial $G$ of degree $e$
$$
G[\bx:z] = g_e(\bx) + z g_{e-1}(\bx) + \ldots + z^e g_0.
$$

\smallskip
Let $\bp$ be a point of $X^\infty\cap Y^\infty$. 
Let $C_\bp (\ovlX)$ and $C_\bp (\ovl{Y})$ be respectively the tangent cones of the germs of
$(\ovlX,\bp)$ and $(\ovl{Y,\bp})$. 
\begin{lemma}\label{lem:cone-tangent-1}
Assume that $C_\bp(\ovlX) \cap C_\bp(\ovl{Y})$ is not contained in $T_\bp\bH = \C^{n-1}\times 0$. 
Then $\dist(X,Y) =0$.
\end{lemma}
\begin{proof}
If the germ $(X\cap Y,\bp)$ is not empty then the result is true.

\smallskip
Assume that $(X \cap Y,\bp)$ is empty, that is $\ovlX\cap\ovl{Y}$ is contained in $\bH$.

Let $\vp,\gm :[0,\ve[\mapsto\cU_1$ be semi-algebraic and real Puiseux arcs such that 
\begin{enumerate}
\item[(i)] $\vp(0) = \gm(0) = \bp$; 
\item[(ii)] $\vp(]0,\ve[) \subset X$ and $\gm(]0,\ve[)\subset Y$; 
\item[(iii)]$\vp'(0) = \gm'(0) = \xi = (\bw,1) \in C_\bp(\ovlX) \cap C_\bp(\ovl{Y})$.
\end{enumerate}
In the chart $\cU_1$ we can semi-algebraically re-parameterize the arcs as 
$$
\vp(t) = (t\bw + t^{1+a}A(t),t)  \in\C^{n-1}\times\C  \mbox{ and }  
\gm(t) = (t\bw + t^{1+b}B(t),t)
$$
with $a>0, A(0)\neq 0$ or $A\equiv 0$, and $b>0,B(0)\neq 0$ or $B\equiv 0$.
Note that we cannot have simultaneously $A \equiv B \equiv 0$.
Thus when $\cU_1$ is endowed with the canonical Euclidean metric 
we know that 
$$
\vp(t) - \gm(t) = t^{1+c}C(t) \in \C^{n-1}\times 0
$$
with $c\geq \min(a,b)>0$ and $C(0)\neq 0$.

We find converging real Puiseux series $A,B$, such that we can also write 

\smallskip
\begin{center}
$
\bar{\vp}(t) = [1:\vp(t)] = [s:\bw+ s^{-a}A(s^{-1}):1]\hfill$ 

$\hfill$ and $  \bar{\gm}(t) := [1:\gm(t)] = [s:\bw + s^{-b}B(s^{-1}):1]
$
\end{center}
for $s = t^{-1}$ so that when $s\to+\infty$ we deduce 
$$
|\bar{\vp}(s) - \bar{\gm}(s)| = s^{-c}C(s^{-1})
$$
for $C$ a converging real Puiseux series.
\end{proof}
We gave the result we just needed here, but the same proof shows that Lemma \ref{lem:cone-tangent-1} stays valid for a
much larger class of subsets. Indeed we have the following:

\smallskip\noindent
{\bf Remark.} \em Let $X,Y$ be two semi-algebraic subsets of $\R^n$. Let $\ovlX$ and $\ovlY$ be the respective 
closures of $X$ and $Y$ taken in $\bP\R^n = \R^n\sqcup \bP\R^{n-1}$ (or in the closed unit ball $\bB_1^n = \R^n \sqcup\bS^{n-1}$
via the Nash embedding $\bx \mapsto \frac{\bx}{\sqrt{1+|\bx|^2}}$). 
Assume there exists $\bp\in \bP\R^{n-1}$ (or in $\bS^{n-1}$) such that $\bp$ lies in $\ovlX \cap \ovlY$.
If the intersection $C_\bp\ovl X\cap C_\bp \ovlY$ of the tangent cones at $\bp$ of the germs
$(\ovlX,\bp)$ and $(\ovlY,\bp)$ is not contained in $T_\bp \bP\R^{n-1}$ (or in $T_\bp \bS^{n-1}$)
then 
$$
\dist (X,Y) = 0.
$$
\em
%
%
%
%
%
%
%
%
%
%
%
%
%
%
%
%
%
%
%
%
%
%
%
%
%
%
%
%
%
%
%
%
%
\section{The Plane case}\label{section:plane}

This section deals with the main result in the plane case, from which the general case will be easily deduced.

\medskip
The section is devoted to show the following
\begin{theorem}\label{thm:main-dim2}
Let $f:\C^2\mapsto\C$ be a non-constant complex polynomial of degree $d$. 
The function $f$ admits a bi-Lipschitz trivial value if and only if it is 
a polynomial in one variable.
\end{theorem}
The rest of the section consists of three Lemmas dealing with all possible situations 
thus proving Theorem \ref{thm:main-dim2}.

\bigskip
Let $F:\C^3\mapsto \C$ be the homogeneous polynomial of degree $d$, such that $f(\bx) = F(\bx,1)$
where $\bx =(x_1,x_2)$.
We write 
$$
F(\bx,z) = f_d(\bx) + zf_{d-1}(\bx) + \ldots + z^{d-1}f_1(\bx) +z^d f_0(\bx).
$$

\medskip
Denoting $z = x_0$, let $\A_i\subset \bP\C^2$ be the  affine chart $\{x_i \neq 0\}$ for 
$i=1,2,0$.

\bigskip
For each $t\in \C$, let $\ovlX_t$ be the projective closure of the level 
$X_t:=\{f = t\} \subset \A_0$, that is 
$$
\ovlX_t = \{F(\bx,z) - tz^d = 0\} \subset \bP\C^2.
$$
For each $t\in \C$, the intersection $\ovlX_t\cap \bH$ is independent of $f$ 
and is equal to 
$$
X^\infty = \{f_d = 0\} \subset \bH = \bP\C^1.
$$

If the degree of the polynomial $f$ is $1$, it is an affine function, therefore 
of a single variable, and the result is trivially true.

\medskip
For the rest of the section, we further assume the following:
\begin{itemize}
\item[(i)] The degree $d$ is equal to $2$ or is larger;
\item[(ii)] By Lemma \ref{lem:cone-tangent-1}, for any point $\bp$ of $X^\infty \cap \bH$, 
we require that either $C_\bp (\ovlX_t) = T_\bp H$ for all but finitely many $t\in \C$, 
or $C_\bp(\ovlX_t) \cap C_\bp (\ovlX_{t'}) = \{\bp\}$ for all $t\neq t'$ but finitely many.
\item[(iii)] Up to a linear change of variables in  $\C^2$, the point $\bp_0 := [1:0:0]$ belongs to $X^\infty$.
\end{itemize}

\medskip
Point (iii) implies that there exist a positive integer $a$ and a homogeneous polynomial $g_d$
of degree $d-a$ such that 
$$
f_d(\bx) = x_2^{a} \cdot g_d  \mbox{ with }  g_d(1,0) \neq 0.
$$ 

Let $[1:u:v] = [\bx:1]$ be local coordinates at $\bp_0$, so that for each $t$, we define the polynomial 
$$
g_t (u,v) := F(1,u,v) - tv^d
$$
defined on the affine chart $\A_1 = \C^2$.
Let $m$ be the multiplicity of $F(1,u,v)$ at $\bp_0$ which is at most equal to $d$. 
\begin{lemma}\label{lem:m=d}
If $m=d$, the polynomial $f$ is a polynomial in the single variable $x_2$.
\end{lemma}
\begin{proof}
The hypothesis implies that for each $k=1,\ldots,d,$ we have $f_k(1,u) = \alpha_k u^k$ for some real
number $\alpha_k$. Thus the result.
\end{proof}

\medskip
We are left now with the case $m < d$. 
\begin{lemma}\label{lem:m<d}
Assume that $m<d$. There exist at most finitely many parameters $t_1,\ldots,t_N,$ such
that for any $t,t'$ different from those, we find
$$
\dist (X_t,X_{t'}) = 0.
$$
\end{lemma}
The proof of Lemma \ref{lem:m<d} will fill the rest of this section.
\begin{proof}[Proof of Lemma \ref{lem:m<d}]
Complex polynomial map germs of bounded degree admits finitely many contact-equivalence
classes, (see Nishimura\cite{Nis}). Besides, the zero loci germs of any two contact-equivalent complex polynomial map germs
have the same \em  embedded topological type: \em a local homeomorphism maps one zero locus germ onto the other one. 
Considering the algebraic family of complex polynomial function germs $(g_t)_{t\in\C}$ at $\bp_0$ (of degree at most $d$),
we deduce that there exist finitely many parameters $t_1,\ldots,t_N,$ such that the embedded topological type of the plane 
curve germ $(\ovlX_t,\bp_0)$ is constant for any $t$ different from the latter ones.
Let $\cO$ be the complementary set of $t_1,\ldots,t_N$.
For each parameter $t$ of $\cO$, the germ
$(\ovlX_t,\bp_0)$ has exactly $s\geq 1$ branches. 
\begin{claim}\label{claim:reduced}
For each parameter $t$ of $\cO$, 
there exist $g_{t,1},\ldots, g_{t,s},$ irreducible function germs, providing
reduced equations to each branch of $(\ovlX_t,\bp_0)$, such that 
$$
g_t : = g_{t,1}\cdots g_{t,s}\, .
$$
\end{claim}
\begin{proof}
The polynomials $g_t$ are defined over the whole affine chart $\A_1$ of $\bP\C^2$. We recall that $x_1 v = 1$, so that so that we have 
$\bH \cap \A_1 = \{v= 0 \}$ and
$$
f(x_1,x_2) - t = v^{-d} g_t(u,v). 
$$
For any parameter $t$ of $\cO$, there exist $g_{t,1},\ldots, g_{t,s},$ irreducible function germs, providing
reduced equations to each branch of $(\ovlX_t,\bp_0)$, and positive integers $\chi_{t,1},\ldots,\chi_{t,s},$ such that 
$$
g_t : = g_{t,1}^{\chi_{t,1}}\cdots g_{t,s}^{\chi_{t,s}} \, .
$$
\\
Assume that for each $t$ that the germ $g_t$ is not reduced.
Since $\du g_t = \du g$, we deduce that $\du g$ is identically zero over a Zariski open set of 
$\A_1$, therefore
$$
\du g \equiv 0.
$$
Then $g$ depends only on $v$, that is $f$ depends only on the variable $x_1$ which is a contradiction.
We deduce that the subset $Z$ of $\C$ of parameters $t$ where $g_t$ reduced is Zariski dense. 
\\
For $t$ in $\cO$, let $X_{t,j}$ be the germ at $\bp_0$ of the branch $\{g_{t,j} = 0\}$, $j=1,\ldots,s$.
Since the embedded topological type of the curve germ $(\ovlX_t,\bp_0) = \cup_{j=1}^s X_{t,j}$ is independent 
of $t\in \cO$, with these notations we can further assume 
that for each $t,t'\in \cO$ and each $j = 1,\ldots,s$, the pairs of germs $(\A_1,X_{t,j},\bp_0)$ and $(\A_1,X_{t',j},\bp_0)$ 
are homeomorphic, so that they have the same Puiseux pairs, therefore same multiplicity $\mu_j$.
\\
For any parameter $t \in Z \setminus \{t_1,\ldots,t_M\}$, we have 
$$
\mu_1 + \cdots + \mu_s = m.
$$
For every parameter $t$ of $\cO$,  we deduce that
$$
\chi_{t,1} = \cdots = \chi_{t,s} = 1.
$$
\end{proof} 

\medskip\noindent
{\bf Hypothesis:} Up to removing finitely many values from $\cO$, we can assume that any parameter $t$ of $\cO$
is also a regular value of $f$.

\bigskip
The proof splits into two cases: The irreducible one and the non-irreducible one.

\bigskip\noindent
For every $t$, we deduce that 
$$
C_{\bp_0}(\ovlX_t) = T_{\bp_0}H = \C\times 0.
$$
since we can write 
$$
g_t(u,v) = g(u,v) - tv^d = g_m(u,v) + g_{m+1}(u,v) + \ldots + g_d(u,v) - tv^d,
$$
we deduce that there exists $\lambda\in\C^*$ such that
$$
g_m (u,v) = \lambda v^m.
$$
In particular we can also write
\begin{eqnarray*}
g(u,v) & = & v^m G(u,v) + P(u,v) \\
& = & v^mG(u,v) + \sum_{k=1}^m v^{m-k} f_{d-m+k}(1,u) \\
& = & v^mG(u,v) + \sum_1^m v^{m-k}u^{(k+a_k)} h_k(1,u) 
\end{eqnarray*}
with $G(0,0) = \lambda$, $h_m$ is a (local analytic) unit at $u = 0$, and $a_k$ 
is a positive integer whenever $h_k$ is non-zero (and 
thus is local analytic unit) for $k=1,\ldots,m-1$.

\medskip\noindent
We recall the following elementary
\begin{claim}\label{lem:key}
Let $m,\beta$ be positive integer numbers. Let $\psi \in \C\{u\}$ 
be a complex function germ at $0\in \C$ of the form $\psi (u) = u^\beta [b_0 + u(\cdots)]$
with $b_0\neq 0$.
Let $\omg$ be a $m$-th primitive root of unity.
Let $P$ be the Weierstrass polynomial defined as
$$
P(u,v) = \Pi_{i=1}^m (v - \psi(\omg^i u)).
$$
Therefore we find 
$$
P(u,v) = v^m + v^{m-1}u^\beta\sgm_1(u)+ \ldots + v^{(m-1)\beta}\sgm_{m-1}(u) + u^{m\beta}\sgm_m(u).
$$
where 
\noindent
\begin{enumerate}
\item $\sgm_m(0) = b_0^m$ 
\item $\sgm_j (u) = u^{\eta_j}\tau_j(u)$ with $\eta_j \geq 1$ and $\tau_j\in\C\{u\}$ for $j=1,\ldots,m-1$.
\end{enumerate}
\end{claim}

\begin{proof}
It is just a consequence of the fact that the symmetric functions of the roots of the polynomial 
$Y^m - 1$ are zero, but for the $0$-th and the $m$-th ones, that is, for $j=1,\ldots,m-1,$
$$
\sum_{1\leq k_1 < \ldots < k_j\leq m} \omg^{k_1 + \ldots + k_j} = 0
$$
\end{proof}

\medskip\noindent
{\bf Case 1:} \em Irreducible case. \em

\smallskip
We assume without loss of generality that $f(0) = 0$ and belongs to $\cO$. It just makes the computations lighter to present.

\smallskip
We recall that, the germ $(\ovl{X}_t,\bp_0)$ has constant embedded topological type
and is irreducible for any parameter $t$ of $\cO$. For each such a generic $t$, we write
\begin{eqnarray*}
g_t(u^m,v) &  =  & g(u^m,v) - tv^d  =  U_t(u,v)\Pi_{i=1}^m(v - \psi_t(\omg^i u)) \\
& = & U_t \cdot S_t(u,v) 
\end{eqnarray*}
for $U_t$ a unit at $\bp_0$ and where $\psi_t \in \C\{u\}$ is the Puiseux root of $g_t$.  We find that 
$$
U_t(u,v) = \sum_{k\geq 0} U_{t,k}(u)v^k  \mbox{ and }  S_t(u,v) = \sum_{l=0}^m v^{m-l}S_{t,l}(u).
$$
In particular we get $S_{t,0} = 1$ and $U_{t,0} = \lambda$. Since 
$$
g_t (u,v) = tv^d + \sum_{j=0}^{d-1} v^j G_j(u); 
$$
and let us denote $G_d = t$ and $G_j \equiv 0$ whenever $j\leq -1$. 
With the convention that $S_{t,l} \equiv 0$ for any $l\leq -1$, we deduce that for each $j\geq 0$

\smallskip
\begin{equation}\label{eq:puiseux-root-t}
\sum_{k+(m-l) = j} U_{t,k}(u)S_{t,l}(u) = G_j(u^m)
\end{equation}
By hypothesis, the embedded topological type of the branch $(\ovlX_t,\bp_0)$ is constant, so that, equivalently,
the Puiseux pairs of $\psi_t$ are independent of $t\in\cO$ by \cite{Bur,Zar} (actually are bi-Lipschitz embedded-ly
invariant \cite{PhTe,Fer}). We find
$$
\psi_t (u) = \sum_{p=1}^s u^{\beta_p}\psi_{t,p}(u^{\alpha_p}) = u^{\beta}[b_{t,0} + \sum_{p\geq 1} b_{t,k}u^{\gm_k}]
$$
with $1 \leq \gm_1 < \ldots < \gm_p < \gm_{p+1} < \ldots$ are integer exponents. Note that each $\psi_{t,p}$ is an analytic unit at $u=0$.
For $j =0$ in Equation (\ref{eq:puiseux-root-t}) we get
\begin{eqnarray*}
G_0 (u^m) & = & f_d(1,u^m) = u^{m(a_m + m)} h_m(u^m) = U_{t,0}(u)S_{t,m}(u) \\
& = & [\lambda + \sum_{p\geq 1} v_{t,p}u^{r_p}]\cdot  (\Pi_{i=1}^m (-\psi_t(\omg^i u)))
\end{eqnarray*}
with either $\sum_p v_{t,p} u^{r_p}$ identically null or each $v_{t,p} \neq 0$ and $1 \leq r_1 < \ldots < r_k < \ldots$
are integer exponents. Then we get
$$
d \geq m + a_m = \beta , \mbox{ and }  h_m(0) = \lambda\omg^\frac{m(m+1)}{2}(-b_{t,0})^m = \lambda b_{t,0}^m
$$
so that for each $t,t'$ in $\cO$, there exists $0\leq q \leq m-1$ such that
\begin{center}
$b_{t',0} = \omg^q b_{t,0}$
\end{center}
Thus we can assume $b_{t,0} = b_{t',0} = b_0$.

Since for all $t$ in the neighbourhood $\cO$ the germs $(\ovl{X_t},\bp_0)$ are irreducible with constant
embedded topological type, we can assume that for all $t$ the root $\psi_t$ writes
\begin{equation}\label{eq:kappa}
\psi_t(u) = u^{\beta}B_0(u) + u^{\beta +\kp_t}A(t,u) \in \C\{u\}
\end{equation}

where $B_0$ and $A_t : = A(t,-)$ are analytic units, with $\psi_0 = u^\beta B_0$, and $\kp_t$ is a positive integer for
$t$ lying in $\cO\setminus 0$.
Since 
$$
g(u^m,v) = v^mG(u^m,v) + \sum_1^m v^{m-k}u^{m(k+a_k)} h_k(1,u^m) 
$$
we deduce that for each $k = 1, \ldots , m-1$ that 
\begin{center}
$(m-k)\beta + m(k+a_k) \geq \beta m,$ equivalently $m (a_k+k) \geq k \beta$
\end{center}

Since $d \geq m+a_k = (m-k) + (k+a_k)$ for each $k=1,\ldots,m-1,$ we deduce whenever $m\geq 2$ that 
\begin{equation}\label{eq:d>beta}
d > \beta
\end{equation}

By Lemma \ref{lem:key} we deduce that for each $t$ of $\cO$ and for each $k=1,\ldots,m-1$, we find
\begin{equation}
S_{l,t}(u) = b_0^{m-l}u^{(m-l+1)\beta + \mu_l}\sgm_{t,l}(u) 
\end{equation}

with $\mu_l\in\N_{\geq 1}$, where  $\sgm_{t,l}$ is analytic for all $t$ of $\cO$, and either a local unit or identically zero for
all $t$ in $\cO \setminus 0$.

Let us write $\psi$ for $\psi_0$ and $U_j$ for $U_{0,j}$ and $S_k$ for $S_{0,k}$.
Since for $j=1,\ldots m-1,$ we know that 
$$
h_{m-j} (1,u^m) = G_j(u^m) = \sum_{k+(m-l) = j} U_k(u)S_l(u),
$$
we deduce that for $k=1,\ldots,m-1,$
\begin{equation}\label{eq:ak+k-larger}
u^{k+a_k} h_k(1,u^m)  =  u^{k\beta + \eta_k}h_k(1,u^m)
\end{equation}
with $\eta_k\in \N_{\geq 1}$.

\smallskip

Using Equation (\ref{eq:ak+k-larger}) we can write $P(u^m,v) = u^{m\beta} R(u,v)$ where
\begin{center}
\begin{tabular}{rcl}
\bigskip
$R(u,v)$ & $=$ & $h_m (1,u^m) + 
v u^{m(m-1+a_{m-1}-\beta)} h_{m-1} (1,u^m) + \ldots$ \\

\medskip
 &  & {$\hfill +  v^{m-1}u^{m(1+a_1-\beta)} h_1 (1,u^m)$}\\
 
\smallskip
& $=$ & 
$\frac{v}{u^\beta} u^{\eta_{m-1}}h_{m-1} (1,u^m) + \ldots + 
\left(\frac{v}{u^\beta}\right)^{m-1}u^{\eta_1}h_1 (1,u^m)$ 
\end{tabular}
\end{center}

\smallskip\noindent
Let us consider the following "blowing up": 
$$
v = \psi_0(u) + u^\beta w = u^\beta [B_0(u) +w],
$$
so that the (strict transform of the) branch corresponding to $g = 0$ is given as $w=0$ after blowing-up.
Thus we find 
$$
g(u^m,v) = u^{m\beta} \cdot w^\mu \cdot \phi(u,w)
$$
for a local unit $\phi$ and $\mu$ a positive integer. We find

\begin{equation}\label{eq:phi}
w^\mu\phi(u,w) = (B_0 + w)^m G(u^m,u^\beta (B_0 +w)) + R(u,u^\beta(B_0 + w)).
\end{equation}
 
\smallskip\noindent
The function germ $(u,w) \mapsto R(u,v(u,w))$ is analytic in $(u,w)$. We obtain

\smallskip
\begin{center}
$
R(u,u^\beta(B_0 + w)) = h_m(1,u^m) + (B_0 +w) \cdot u^{\eta_{m-1}}h_{m-1}(1,u^m) + \ldots + \hfill$ 

\smallskip
$\hfill (B_0 + w)^{m-1}u^{\eta_1}h_1(1,u^m) 
$
\end{center}
We recall also that 
$$
G(u,v) = G_m(u) + vG_{m+1}(u) + \ldots + v^{d-1}G_{d-1}(u).
$$
Let us examine the coefficient of $w$ in the expression of $\phi$. Writing 
$$
w^\mu\phi (u,w) = w\phi_1(u) + w^2\phi_2(u,w),
$$
an elementary computation from Equation (\ref{eq:phi}) yields
$$
\mu = 1  ,  \mbox{ and }  \phi_1(u) = m b_0^{m-1} \lambda + u \tht_1(u)
$$
for an analytic function $\tht_1$. 
Resolving the equation in $v$
$$
g(u^m,v) = tv^d
$$
turns into resolving the equation in $w$,  
$$
w[mB_0(u) + \tht (u,w)] = t u^{\beta (d-m)}(B_0 +w)^d
$$
for $\tht$ an analytic function germ vanishing at $(0,0)$.
\begin{lemma}\label{lem:kappa}
For all $t$ in the neighbourhood $\cO$, we obtain
\begin{center}
$\kp_t = \frac{(d-m)\beta}{m}$ 
\end{center}
\end{lemma}
\begin{proof}
The equation $g = tv^d$, for $t\neq 0$, has a solution after blowing-up which writes as follows 
$$
w =  u^{\beta(d-m)} B_t(u)
$$ 
where 
$B_t(0) 
 \neq 0$. 
Since $v = u^\beta[B_0(u) + w]$ 
we find $\kp_t = \frac{\beta(d-m)}{m}$.
\end{proof}
To conclude the irreducible case, it remains to check that the distance
between two generic levels of the polynomial $f$ is $0$.

\smallskip\noindent
For generic $t$, we obtained $\kp_t = \frac{\beta(d-m)}{m} = : \kp$ is constant and that 
$A_t(0) \neq 0$. Thus 
$$
\psi_t(u) - \psi (u) = u^{\beta + \kp}A_t(u) 
$$
The mapping $u \mapsto [1:u^m:u^{\beta}B_0(u) + u^{\beta + \kp}A_t(u)] = [x:y_t(x):1]$
is a parameterization at infinity of the branch $\ovlX_t$. Suppose
that $t$ is fixed, then 
$$
x = \frac{1}{u^\beta(B_0(u) + u^\kp A_t(u))}  \mbox{ and }  y_t(x) = \frac{u^m}{u^\beta(B_0(u) + u^\kp A_t(u))}
$$ 
and we deduce 
$$
u = x^\frac{-1}{\beta}[\theta_0 + x^\frac{-\kp}{\beta}\alpha_t]
$$
where $\theta_0$ is an invertible converging power series in $x^\frac{-1}{\beta}$, independent from $t$; 
and $\alpha_t$ is a converging power series in $x^\frac{-1}{\beta}$ with $\alpha_0 \equiv 0$.
We find
$$
y_t(x) = x^\frac{\beta-m}{\beta}[\theta_0^{m-\beta} + x^\frac{-\kp}{\beta}\gamma_t].
$$
and $\gamma_t$ is a converging power series in $x^\frac{-1}{\beta}$ with $\gamma_0 \equiv 0$. Therefore 
$$
y_t(x) - y_0(x) =  x^{\frac{\beta - m}{\beta} + \frac{(m-d)}{m}} \delta_t .
$$
and $\delta_t$ is a converging power series in $x^\frac{-1}{\beta}$.
Since $m < \beta = a_m + m \leq d$, we deduce  
$$
\lim_{x\to \infty} y_t(x) - y_0(x) = 0.
$$

\bigskip\noindent
{\bf Case 2:} \em Non-irreducible case. \em

\medskip\noindent
Claim \ref{claim:reduced} asserts that, for each $t\in\cO$, the function germ $g_t$ is reduced at $\bp_0$, namely
$$
g_t : = g_{t,1}\cdots g_{t,s}  , 
$$
for irreducible function germs $g_{t,1},\ldots, g_{t,s}$.

\smallskip\noindent
For each $i=1,\ldots,s,$ we have
$$
g_{t,i} (u^{m_i},v) = U_{t,i}(u,v) \Pi_{j=1}^s (v - A_{t,i}(\omg_i^ju)) = U_{t,i}(u,v) \cdot P_{t,i}(u,v)
$$
where $U_{t,i}$ is an analytic unit and $A_{t,i}$ is analytic and $A_{t,i}(u) = u^{\beta_i}\cdot Unit$ 
for $(m_i,\beta_i)$ the first Puiseux pair of the branch $\{g_{t,i}=0\}$, and last, $\omg_i$ is a $m_i$-th 
primitive root of unity.

Let $L$ be the lowest common multiple of $m_1,\ldots,m_s,$ so that we have
$L = L_i\cdot m_i$ for each $i=1,\ldots,s$. 
Let 
$$
P_t(u^\frac{1}{L},v):= \Pi_{i=1}^s P_{t,i}(u^\frac{1}{m_i},v)
$$
Thus we deduce
$$
g_t(u,v) = (\Pi_{i=1}^sU_{t,i}(u^\frac{1}{m_i}))\cdot P_t(u^\frac{1}{L},v) = U_t(u^\frac{1}{L},v)\cdot P_t(u^\frac{1}{L},v)
$$
Thus
for each $i$ we write 
$$
P_{t,i}(u^\frac{1}{m_i}) = v^{m_i} + \sum_{j=1}^{m_i-1} v^{m_i-j} \cdot u^{j\frac{\beta_i}{m_i}+e_{t,i,j}} 
\cdot S_{t,i,j}(u^\frac{1}{m_i}) + u^{\beta_i} \cdot S_{t,i,m_i}(u^\frac{1}{m_i})
$$
where $S_{t,i,j}$ is either identically zero or an analytic unit and $S_{t,i,m_i}(0)\neq 0$.
Let 
$$
u_L := u^\frac{1}{L} \mbox{ and } \beta := \sum_{i=1}^s \beta_i.
$$ 
From Lemma \ref{lem:key}, we know for $i=1,\ldots,s,$ that 
\begin{enumerate}
\item $e_{t,m_i,i} = 0$  and,
\item $m_ie_{t,k_i,i} \geq 1$ for $1\leq k_i\leq m_i - 1$.
\end{enumerate}

Then
$$
P_t(u_L,v) = v^m + \sum_{j=1}^m v^{m-j} \cdot u_L^{p_{t,j}} \cdot S_{t,j}(u_L) + u_L^{\beta}S_{t,m}(u_L)
$$
where $S_{t,j}$ is either identically zero or an analytic unit, in which
case $p_{t,j}$ is a positive integer, and $S_{t,m}(0)\neq 0$.
Let again 
$$
I_{m-j} := \{\bk \in\N^s :  m_i \geq k_i \geq 0 
\mbox{ for }  i=1,\ldots,s ,  \mbox{ and } \sum_{i=1}^s k_i = j\}
$$
Let $\bb$ be the vector of $\Q^s$ with coordinates $b_i = \frac{\beta_i}{m_i}$.
For $\bk = (k_1,\ldots,k_s)\in I_{m-j}$, let 
$$E_{t,\bk} := \sum_{i=1}^s e_{t,k_i,i}   
\mbox{ and } \langle\bk,\bb\rangle := \sum_{i=1}^s k_i\cdot b_i.
$$
We observe that for $j=1,\ldots,m-1,$ that 
$$
E_{t,\bk} > 0.
$$
We deduce for $j=1,\ldots,m-1,$
$$
u_L^{p_{t,j}} \cdot S_{t,j}(u_L) 
= \sum_{\bk\in I_{m-j}} u^{\langle\bk,\bb\rangle + E_{t,\bk}}\cdot\Pi_{i=1}^s S_{t,i,k_i}(u^\frac{1}{m_i})
$$
with the convention that $S_{t,i,0} := 1$. 

\medskip\noindent
{\bf Hypothesis:} $\frac{\beta_1}{m_1} \leq \frac{\beta_2}{m_2} \leq \ldots \leq \frac{\beta_s}{m_s}$.

\medskip
In order to avoid discussing convergence systematically, we will work in $\C[[u^\frac{1}{L},v]]$. We get 
$$
U_t = \sum_{k\geq 0} C_{t,k}(u_L) \cdot u_L^{n_{t,k}}\cdot v^k
$$
with $C_0 (0) = \lambda$ and $n_0=0$, $C_{t,k}$ either identically zero or a unit, with, in this latter case, $n_{t,k} \in \N$.
Since $g - tv^d = U_t\cdot P_t$ we deduce that for each $k=1,\ldots,m-1,$ 
$$
u^{k+a_k}\cdot h_k(1,u) = \sum_{j=k}^m u^\frac{p_{t,j} + n_{t,j-k}}{L} \cdot S_{t,k}\cdot C_{t,j-k}
$$
and thus 
$$
k+a_k \geq \min_{j=k,\ldots,m} \frac{p_{t,j}}{L} \geq \min_{j=k,\ldots,m} (\min_{\bk\in I_{m-j}} \langle \bk,\bb\rangle 
+ E_{t,\bk})
$$

\noindent
For each $k=1,\ldots,m-1$, let $\bk_0 = (k_1,\ldots,k_s)\in \N^s$ such that $k_i\leq m_i$ and
$$
\min_{j\geq k}  ( \min_{\bk\in I_{m - j}} \langle \bk,\bb \rangle + E_{t,\bk} ) =
\langle \bk_0,\bb\rangle + E_{t,\bk_0} .
$$
Let $\bm = (m_1,\ldots,m_s)$. We deduce that for $k=1,\ldots,m-1$, there exists $\ve_{t,k} > 0$ such that 

\begin{equation}\label{eq:k+ak}
(m-k)\frac{\beta_s}{m_s} + k + a_k 
\geq \langle \bm-\bk_0,\bb\rangle + \langle \bk_0,\bb\rangle + E_{t,\bk_0} 
= \beta + \ve_{t,k}
\end{equation}
Since $g_t(u,A_{t,s}(u^\frac{1}{m_s})) = 0$, Equation (\ref{eq:k+ak}) yields
$$
m\frac{\beta_s}{m_s} = \beta ,
$$
from which we deduce
$$
\frac{\beta_1}{m_1} = \frac{\beta_i}{m_i}  ,  i = 1,\ldots,s
$$
As a consequence of this fact, we deduce that 
$$
g(u,v) - tv^d = v^m \cdot Unit - u^\beta\cdot Unit - tv^d
$$
where the units are a-priori formal but such that each $Unit(u^{L},v)$ is analytic.
From such an expression we conclude as in the irreducible case.
\end{proof}
%
%
%
%
%
%
%
%
%
%
%
%
%
%
%
%
%
%
%
%
%
%
%
%
%
%
%
%
%
%
%
%
%
%
%
%
%
\section{Main result: General case}

\medskip
The section is devoted to show the main result of this note:
\begin{theorem}\label{thm:main}
Let $f:\Cn\mapsto\C$ be a non constant complex polynomial.
The function $f$ admits a bi-Lipschitz trivial value if and only if it is a polynomial in one variable.
\end{theorem}
This result follows immediately from the following 
\begin{lemma}\label{lem:main}
Let $f:\C^n\mapsto \C$ be a complex polynomial of degree $2$ or larger. 
Assume that there exist a complex value $c$ and a neighbourhood $\cU$ of $c$ in $\C$
such that there exists a positive constant $L$ for which the following hold true:
$$
\dist (f^{-1}(s),f^{-1}(t)) \geq L|s -t|  \mbox{ for any }  s,t \in \cU.
$$
Then $f$ depends only on a single variable.
\end{lemma}
\begin{proof}[Proof of Lemma \ref{lem:main}]
It is sufficient to show that $f$ depends on $n-1$ variables, that is there exists a non zero 
vector $\xi$ of $\C^n$ such that $\xi\cdot f \equiv 0$. Indeed, 
an induction on the dimension of the ambient space will work since we have proved 
Theorem \ref{thm:main-dim2} treating the plane case.

\medskip
Let $F:\C^{n+1}\mapsto \C$ be the homogeneous polynomial of degree $d$, such that $f(\bx) = F(\bx,1)$.
Thus
$$
F(\bx,z) = f_d(\bx) + zf_{d-1}(\bx) + \ldots + z^{d-1}f_1(\bx) + z^d f_0
$$
For each $t\in \C$, let again $\ovlX_t$ be the closure in $\bP\C^n$ of the level $X_t:=\{f = t\}$, and thus, 
$$
\ovlX_t\cap \bH = X^\infty = \{f_d = 0\} \subset \bH.
$$
By Lemma \ref{lem:cone-tangent-1} and the current hypothesis, we must have
\begin{equation}\label{eq:tgt-infty}
C_\bp (\ovlX_t) = T_\bp \bH  \mbox{ for each } t \in \cU.
\end{equation}
After a linear change of variables in $\Cn$, the point $\bp_0 := [1:0:0]$ is a point of $X^\infty$, thus 
of any $\ovlX_t$. 

\bigskip\noindent
{\bf Induction Hypothesis:} \em Assume $n\geq 3$ and Lemma \ref{lem:main} holds true in dimension $2,\ldots,n-1$. \em 

\medskip
Let us write again $\bx = (x_1,\by) \in \C\times\C^{n-1}$. 
In the affine chart $\cU_1:=\{x_1\neq 0\}$ of $\bP\Cn$, with affine coordinates $[1:\bu:v]$, we define 
\begin{eqnarray*}
g_t(\bu,v) & = & F(1,\bu,v) - tv^d \\
& = & f_d(1,\bu) + v f_{d-1}(1,\bu) + \ldots + v^{d-1}f_1(1,\bu) + (f_0 - t)v^d
\end{eqnarray*}
The (possibly non-reduced) affine equation $f(\bx) - t =0$ writes in the chart $\cU_1$
$$
g_t(\bu,v) = 0.
$$

\smallskip\noindent
Let $m$ be the multiplicity of $g_0$ at $\bp_0 = (0,0) \in \C^{n-1}\times \C = \cU_1$.

\smallskip\noindent
$\bullet$ If $m = d$, as in the case $n=2$, we deduce that each $\bu\mapsto f_k(1,\bu)$ is either null or
homogeneous of degree $k$. Which implies that $f(\bx) = f(\by)$.

\smallskip\noindent
$\bullet$ Assume that $m<d$. 
Since 
$$
g_t = g_0 - t v^d = g_m + \ldots + g_d +(f_0 - t)v^d
$$
where each $g_k$, for $k=m,\ldots,d$, is a homogeneous polynomial of degree $k$ and 
independent of $t$, using Lemma \ref{lem:cone-tangent-1}, we deduce as in the case 
$n=2$ that $g_m = \lambda v^m$. Note that 
$$
g_t(\bu,0) = f_d(1,\bu).
$$
Since $f_d(1,0,\ldots,0) = 0$ but $f_d\not\equiv 0$, we can assume that the affine coordinate $x_2$ is such that $f_d(x_1,x_2,0)$ 
is not constant since the multiplicity of $\bu \mapsto f_d(1,\bu)$ at $\bu = 0$ is at least $m+1\geq 2$.

Let $f^o$ be the restriction of $f$ to $\C^2\times 0$ and let $f_k^o$ be the restriction to $\C^2\times 0$ 
of the homogeneous component of degree $k$ of $f$. 
Since $f_d^o$ is homogeneous of degree $d$ and not constant, thus $f^o$ is not constant and of degree $d$.
Let $X_t^o := \{f^o - t\} = X_t\cap \C^2\times 0$.
Let $F^o$ be the restriction of $F$ to $\C^2\times 0 \times \C$, that is
$$
F^o (x_1:x_2:z) = F(x_1:x_2:0:\ldots:0:z).
$$
Let $\bp_0^o= [1:0:0] \in \bP\C^2$. Taking coordinates in the affine chart $x_1\neq 0$ of $\bP\C^2$, 
we get 
\begin{eqnarray*} 
g_t^o(u,v) & = & F^o(1,u,v) - tv^d  \\
& = & f_d^o(1,u) + vf^o_{d-1}(1,u) + \ldots + v{d-1}f_1^o(1,u) + (f_0 - t)v^d
\end{eqnarray*}
In particular we see that at $\bp_0^o$ the multiplicity of $g_t^o$ is $m<d$, and thus from 
the work done in dimension $2$ we deduce that 
$$
\dist ( X_t^o, X_s^o) = 0.
$$
Therefore the case $m<d$ cannot happen. So ends the induction procedure and thus
the proof of Lemma \ref{lem:main}.
\end{proof}
%
%
%
%
%
%
%
%
%
%
%
%
%
%
%
%
%
%
%
%
%
%


%
%
%
%
%
%
%
%
%
%
%
%
%
%
%
%
%
%
%
%
%
%
%
%
%
%
%
%
%

%
%

%
%
\end{document}